\documentclass[a4paper,11pt]{article}
\pdfoutput=1

\usepackage[a4paper,tmargin=3truecm,bmargin=3truecm,rmargin=2.5truecm,
lmargin=2.5truecm,twoside,verbose=true]{geometry}
\usepackage{cancel,graphicx}

\usepackage{amsmath,amssymb}
\usepackage[amsmath, hyperref, thmmarks]{ntheorem}
\usepackage[all]{xypic}
\usepackage[pdftex]{hyperref}
\usepackage[english]{babel}

\numberwithin{equation}{section}

\allowdisplaybreaks[1]

\def\question#1{}% pour poser une question
\newcommand\caB{{\mathcal B}}
\newcommand\caC{{\mathcal C}}

\newcommand\caL{{\mathcal L}}

\newcommand\caS{{\mathcal S}}

\newcommand\caU{{\mathcal U}}

\newcommand\gone{{ \mathchoice {1\mskip-4mu\mathrm{l} } {1\mskip-4mu\mathrm{l} }{1\mskip-4.5mu\mathrm{l} } {1\mskip-5mu\mathrm{l}} }}
\newcommand\gR{{\mathbb R}}

\newcommand\gT{{\mathbb T}}
\newcommand\gC{{\mathbb C}}

\newcommand\gB{{\mathbb B}}
\newcommand\gN{{\mathbb N}}
\newcommand\gZ{{\mathbb Z}}

\newcommand\superA{{\mathcal A}}
\newcommand\algA{{\mathbf A}}

\newcommand\algB{{\mathbf B}}

\newcommand\kg{{\mathfrak g}}

\newcommand\eps{{\varepsilon}}

\newcommand\fois{\mathord{\cdot}}
\newcommand\dd{{\text{\textup{d}}}}

\newcommand\norm{\mathord{\parallel}}
\theoremsymbol{}
\theorembodyfont{\slshape}
\theoremheaderfont{\normalfont\bfseries}
\theoremseparator{}
\newtheorem{Theorem}{Theorem}[section]
\newtheorem{theorem}[Theorem]{Theorem}

\newtheorem{proposition}[Theorem]{Proposition}

\newtheorem{lemma}[Theorem]{Lemma}

\theorembodyfont{\upshape}
\theoremsymbol{\ensuremath{\blacklozenge}}

\newtheorem{example}[Theorem]{Example}

\newtheorem{remark}[Theorem]{Remark}

\newtheorem{definition}[Theorem]{Definition}

\theoremstyle{nonumberplain}
\theoremheaderfont{\scshape}
\theorembodyfont{\normalfont}
\theoremsymbol{\ensuremath{\blacksquare}}

\newtheorem{proof}{Proof}
\qedsymbol{\ensuremath{_\blacksquare}}
\theoremclass{LaTeX}

\renewenvironment{thebibliography}[1]
         {\section*{References}\frenchspacing\small
          \begin{list}{[\arabic{enumi}]}
         {\usecounter{enumi}\parsep=2pt\topsep 0pt
         \settowidth{\labelwidth}{[#1]}
         \leftmargin=\labelwidth\advance\leftmargin\labelsep
         \rightmargin=0pt\itemsep=1pt\sloppy}}{\end{list}}

\title{Fr\'echet Quantum Supergroups\footnote{Work
supported by the Belgian Interuniversity Attraction Pole (IAP) within the framework ``Dynamics, Geometry and Statistical Physics'' (DYGEST).}}
\date{}
\author{Axel de Goursac}

\begin{document}

\maketitle
\vspace*{-1cm}
\begin{center}
\textit{Charg\'e de Recherches au FRS-FNRS,\\ IRMP, Universit\'e Catholique de Louvain,\\ Chemin du Cyclotron 2, 1348 Louvain-la-Neuve, Belgium\\
e-mail: \texttt{axelmg@melix.net}}\\
\end{center}%

\vskip 2cm

\begin{abstract}
In this paper, we introduce Fr\'echet quantum supergroups and their representations. By using the universal deformation formula of the abelian supergroups $\gR^{m|n}$ we construct various classes of Fr\'echet quantum supergroups that are deformation of classical ones. For such quantum supergroups, we find an analog of Kac-Takesaki operators that are superunitary and satisfy the pentagonal relation.
\end{abstract}

\vskip 5cm

\vskip 1cm

\pagebreak

\section{Introduction}

Noncommutative geometry \cite{Connes:1994} is a vibrant field of mathematics whose essential principle lies in the duality between spaces and commutative algebras, so that the properties of spaces can be algebraically characterized. Then, a noncommutative algebra can be seen as corresponding to some ``noncommutative space''. This very rich way of thinking allows generalizing classical notions and theorems of usual geometry, and it is sometimes possible to prove new results for differential geometry in this more general noncommutative framework (for instance the classification of foliations of the torus \cite{Rieffel:1981}). In this point of view, the noncommutative analogs of groups are quantum groups \cite{Woronowicz:1987,Majid:1995}.

As productive examples of noncommutative algebras, deformation quantization \cite{Bayen:1978} consists in introducing a deformed product on the space of smooth functions $\caC^\infty(M)$ on a Poisson manifold $M$. This product depends on a deformation parameter $\theta$ so that $\theta=0$ yields the usual commutative product on $\caC^\infty(M)$. There is thus possibility of studying deformations with a formal deformation parameter (see in particular \cite{Kontsevich:2003}) or a non-formal one ($\theta\in\gR$).

In the case of a symplectic Lie group $G$, to any left-invariant formal deformation on $\caC^\infty(G)$ is associated a Drinfeld twist \cite{Drinfeld:1989} on the universal enveloping Hopf algebra $\caU(\kg)$ of the Lie algebra of $G$. Then, such a twist $F\in\caU(\kg)\otimes\caU(\kg)[[\theta]]$ deforms also any $\caU(\kg)$-module-algebra $\algA$; this is called a universal deformation formula (UDF). The external symmetries of the UDF correspond thus to the twisted Hopf algebra on which the deformation of the algebras $\algA$ are module-algebras (see \cite{Giaquinto:1998}).

For non-formal deformation quantization of Lie groups in the smooth setting, there are only few available examples. Rieffel \cite{Rieffel:1989} built the deformation of Abelian groups and the associated UDF. This was also recently extended to non-Abelian K\"ahlerian Lie groups \cite{Bieliavsky:2010kg,Bieliavsky:2013sk}.

\medskip

Coming from another direction, supergeometry \cite{Kostant:1977,Tuynman:2005} is a mathematical theory in which the objects are supermanifolds involving, besides the usual commuting coordinates, also anticommuting coordinates (Grassmann variables). The algebra of smooth functions of a supermanifold is then $\gZ_2$-graded commutative. Supergeometry was applied to various domains of mathematics and in physics.

It is then natural to ask whether a noncommutative supergeometry corresponding to noncommutative geometry with $\gZ_2$-grading does exist and possess nice properties. Noncommutative algebraic geometry developed fruitfully this graded approach with projective schemes \cite{Artin:1994}. A work in the direction of noncommutative Q-manifolds was also achieved in \cite{Schwarz:2000}. In \cite{deGoursac:2008bd} we built some geometric tools such as noncommutative differential calculi, connections, for algebras with more general grading and interpreted as ``noncommutative graded spaces''. More recently, we constructed a non-formal deformation quantization of Abelian Lie supergroups in \cite{Bieliavsky:2010su}. It was initially motivated by physics since a renormalizable scalar quantum field theory on the Moyal space can be interpreted with the star-product of the superspace $\gR^{m|1}$ (see \cite{deGoursac:2010zb,Bieliavsky:2010su}), as well as its associated gauge theory \cite{deGoursac:2008bd,deGoursac:2007gq}. In this deformation, we had to introduce the notion of C*-superalgebra in order to implement the UDF associated to the Heisenberg supergroup. This notion has nice properties and should be the natural object of noncommutative supergeometry at the topological level.

\medskip

The corresponding notion of quantum group in noncommutative supergeometry should be called ``quantum supergroup''. Some algebraic definitions of quantum supergroups were already given (see e.g. \cite{Majid:1995}). In this paper, we introduce this notion in the context of topological Hopf superalgebras.

To this aim, we first look at the external symmetries of the UDF associated to the deformation of the Abelian Lie supergroups. We indeed find a non-nuclear Fr\'echet-Hopf superalgebra $H$ whose comodule-algebras are deformed by the twist of the UDF and which corresponds to the external symmetries.

As external symmetries form quantum groups in general, properties of $H$ lead us to a Fr\'echet definition of quantum supergroups and of their representations. This definition is actually a direct extension of Kostant's definition \cite{Kostant:1977} of supergroups without the supercommutativity condition.

We then study three examples of Fr\'echet quantum supergroups. First, the Clifford algebra that is topologically trivial as finite-dimensional. The second example uses the UDF of the Abelian Lie supergroups to deform a class of solvable (non-nilpotent) Lie supergroups into Fr\'echet quantum supergroups. We introduce an analog of Kac-Takesaki operator for such quantum supergroups and show that it satisfies the pentagonal equation, but it is superunitary and not unitary. Finally, we construct Fr\'echet quantum supergroups with supertoral subgroups and exhibit their multiplicative superunitary operators.

Note that the definition and properties of C*-quantum supergroups are currently under study, but the Fr\'echet framework presented here - even though not nuclear - is much less constrained and could be useful in some cases where the C* notion is not available.

\section{Non-formal deformation of superspaces}

\subsection{Supergeometric setting}
\label{subsec-concrete}

We start with some recalls about the concrete approach of supergeometry developed by DeWitt, Rogers, Tuynman,... (see \cite{DeWitt:1984,Rogers:2007,Tuynman:2005}). The essence of this approach consists of replacing the basis field $\gR$ by a real supercommutative superalgebra $\superA$ in all the geometric constructions.

Let $\superA=\bigwedge V$, where $V$ is a real infinite-dimensional vector space. Then, $\superA=\superA_0\oplus\superA_1$ is a $\gZ_2$-graded commutative algebra with
\begin{equation*}
\forall a,b\in\superA\quad:\quad ab=(-1)^{|a||b|}ba,
\end{equation*}
where $|a|\in\gZ_2$ denotes the degree of the homogeneous element $a$, and the expression is extended by linearity to inhomogeneous elements of $\superA$. Moreover, it satisfies $\superA/\mathcal{N}_{\superA}\simeq \gR$, where $\mathcal{N}_{\superA}$ denotes the ideal of nilpotent elements of $\superA$. We denote by $\gB:\superA\to\gR$ the quotient map by $\mathcal{N}_{\superA}$, and call it the body map. Actually, the explicit form of the algebra $\superA$ is not important here, only its above properties play a role. Moreover, no topology is needed for $\superA$ here, the Fr\'echet topology will appear at the level of the superfunctions on the involved supermanifolds.

\begin{definition}[superspace]
\label{def-concrete-superspace}
The superspace of (graded) dimension $m|n$ is defined as $\gR^{m|n}:=(\superA_0)^m\times(\superA_1)^n$. It involves $m$ even (commuting) coordinates and $n$ odd (anticommuting) coordinates in the canonical basis. The body map can be applied on each even coordinate and is also denoted by $\gB:\gR^{m|n}\to\gR^m$.

Moreover, if $m$ is even, this superspace can be endowed by the even symplectic structure associated to the matrix $\omega=\begin{pmatrix} \omega_0 & 0 \\ 0 & 2\gone \end{pmatrix}$ of size $(m+n)$ in the canonical basis, with $\omega_0=\begin{pmatrix} 0 & \gone \\ -\gone & 0 \end{pmatrix}$ of size $m$.
\end{definition}

The DeWitt topology of $\gR^{m|n}$ can be constructed as follows. A subset $U$ of $\gR^{m|n}$ is called open if $\gB U$ is an open subset of $\gR^m$ and $U=\gB^{-1}(\gB U)$, namely $U$ is saturated with nilpotent elements. It is of course not a Hausdorff topology.

The smooth functions on $\gR^{m|0}=(\superA_0)^m$ can be defined as associated to elements of $\caC^\infty(\gR^m)$.
\begin{definition}
\label{def-superextofsmoothfunc}
To any smooth function $f\in \caC^\infty(\gR^m)$ one can associate the function $\tilde f:\gR^{m|0}\to\superA_0$ defined by: $\forall x\in\gR^{m|0}=(\superA_0)^m$, with $x=x_0+n$, $x_0=\gB(x)\in\gR^m$ and $n\in\gR^{m|0}$ a nilpotent element,
\begin{equation*}
\tilde f(x)=\sum_{\alpha\in\gN^m}\frac{1}{\alpha !}\partial^\alpha f(x_0) n^\alpha,
\end{equation*}
with the usual notations for the multi-index $\alpha$. Note that the sum over $\alpha$ is finite due to the nilpotency of $n$.
\end{definition}

\begin{definition}[smooth superfunctions]
Let $U$ be an open subset of $\gR^{m|n}$. A map $f:U\to\superA$ is said to be smooth on $U$, and written $f\in \caC^\infty(U)$, if there exist unique functions $f_I\in \caC^\infty(\gB U)$ for all ordered subsets $I$ of $\{1,\dots,n\}$, such that $\forall (x,\xi)\in\gR^{m|n}$ ($x\in\gR^{m|0}$ and $\xi\in\gR^{0|n}$),
\begin{equation*}
f(x,\xi)=\sum_{I}\tilde{f_I}(x)\xi^I,
\end{equation*}
where $\xi^I$ denotes the ordered product of the corresponding coefficients. This means that, if $I=\{i_1, \dots, i_k\}$ with $1 < i_1 < \cdots < i_k\le n$, then $\xi^I:=\prod_{i\in I}\xi^i =\xi^{i_1} \xi^{i_2} \cdots \xi^{i_k}$, and we take as a convention: $\xi^\emptyset = 1$. We extend this definition in the usual way to functions with values in a superspace.
\end{definition}
For any two (ordered) subsets $I = \{i_1, \dots, i_l\}$ and $J = \{j_1, \dots, j_\ell\}$ of $\{1,\dots,n\}$ we define $\eps(I,J)$ to be zero if $I$ and $J$ overlap; if $I\cap J=\emptyset$, we set $\eps(I,J)$ to the parity of the list $(i_1, \dots, i_k, j_1, \dots, j_\ell)$, defined as $-1$ raised to the number of transpositions needed to put it in increasing order. This function satisfies
\begin{equation}
\eps(I,J)=(-1)^{|I||J|}\eps(J,I),\qquad \eps(I,J\cup K)=\eps(I,J)\eps(I,K)\text{ if }J\cap K=\emptyset.\label{eq-eps}
\end{equation}
As a consequence, we have: $\xi^I\fois\xi^J=\eps(I,J)\xi^{I\cup J}$. The smooth superfunctions then satisfy $\caC^\infty(\gR^{m|n})\simeq \caC^\infty(\gR^m)\otimes\bigwedge \gR^n$. We recall the Lebesgue-Berezin integration for superfunctions:
\begin{equation*}
\int_{\gR^{m|n}}\dd z\, f(z)=\int_{\gR^m}\dd x\,f_{\{1,\dots,n\}}(x).
\end{equation*}

With this definition of smooth superfunctions and the DeWitt topology, it is possible to define supermanifolds and Lie supergroups (see \cite{DeWitt:1984,Rogers:2007,Tuynman:2005}).
\begin{definition}[supermanifold, Lie supergroup]
\label{def-concrete-superman}
Let $M$ be a topological space.
\begin{itemize}
\item A chart of $M$ is a homeomorphism $\varphi:U\to W$, with $U$ an open subset of $M$ and $W$ an open subset of $\gR^{m|n}$, for $m,n\in\gN$.
\item An atlas of $M$ is a collection of charts $\caS=\{\varphi_i:U_i\to W_i,\, i\in I\}$ where $\bigcup_{i\in I}U_i=M$ and $\forall i,j\in I$, $\varphi_i\circ\varphi_j^{-1}\in \caC^\infty(\varphi_j(U_i\cap U_j),W_i)_0$.
\item If $M$ is endowed with an atlas, we define its body as:
\begin{equation*}
\gB M=\{y\in M,\, \exists \varphi_i/\, y\in U_i\text{ and }\varphi_i(y)\in\gB W_i\},
\end{equation*}
and the body map $\gB:M\to\gB M$ on each subset $U_i$ by: $\gB_{|U_i}=\varphi_i^{-1}\circ \gB\circ \varphi_i$.
\item $M$ is called a supermanifold if it is endowed with an atlas such that $\gB M$ is a real manifold. 
\item Let $M$ be a supermanifold. A function $f$ on $M$ is called smooth, and written $f\in \caC^\infty(M)$, if for any chart $\varphi_i$ in an atlas for $M$, $f\circ\varphi_i^{-1}\in \caC^\infty(W_i)$.
\item A Lie supergroup is a supermanifold $G$ which has a group structure for which the multiplication is a smooth map. Consequently, the identity element of the supergroup has real coordinates (it lies in $\gB G$), and the inverse map is smooth.
\end{itemize}
\end{definition}

The algebra $\caC^\infty(M)$ of smooth superfunctions on a supermanifold $M$ carries a structure of $\gZ_2$-graded Fr\'echet superalgebra for the pointwise product (see Lemma 2.18 of \cite{Bieliavsky:2010su}). A supermanifold $M$ of dimension $m|n$ is called trivial if there exists a supermanifold $M_0$ of dimension $m|0$ such that $M\simeq M_0\times\gR^{0|n}$. Note that $\gB M_0=\gB M$ and that $M_0$ is totally determined by $\gB M$. In particular, it can be showed (see \cite{Tuynman:2005}) that every Lie supergroup has an underlying structure of trivial supermanifold.

Note that the superspace $\gR^{m|n}$ has a structure of Abelian supergroup. Its law can be expressed as
\begin{equation*}
\forall (x,\xi),(y,\eta)\in\gR^{m|n}\quad:\quad (x,\xi)\fois (y,\eta)=(x+y,\xi+\eta).
\end{equation*}

\subsection{The star-product}
\label{subsec-prod}

The construction of the deformation quantization of the symplectic superspace $\gR^{m|n}$ (see Definition \ref{def-concrete-superspace}) has been performed in \cite{Bieliavsky:2010su} if $m$ is an even integer. Let us recall here the corresponding $\gR^{m|n}$-invariant star-product. Its expression is given by the von Neumann formula extended to the graded setting: for $x\in \gR^{m|0}$, $\xi\in\gR^{0|n}$ (we write $(x,\xi)\in\gR^{m|n}$),
\begin{multline}
(f_1\star f_2)(x,\xi)=\kappa\int\dd x_1\dd\xi_1\dd x_2\dd\xi_2\ f_1(x_1,\xi_1)f_2(x_2,\xi_2)\\
e^{\frac{-2i}{\theta}(\omega_0(x_1,x_2)+\omega_0(x_2,x)+\omega_0(x,x_1)+2\xi_1\xi_2+2\xi_2\xi+2\xi\xi_1)},\label{eq-prod-moy}
\end{multline}
where $\kappa=(-1)^{\frac{n(n+1)}{2}}\frac{(i\theta)^n}{4^n(\pi\theta)^m}$ is a normalization factor while $\theta$ is the deformation parameter.

This product is defined on smooth superfunctions with compact support (i.e. its body support is compact) but it is possible to extend it to a larger algebra by using the method of oscillatory integrals. Let us introduce the space
\begin{equation*}
\caB(\gR^{m|n}) =\caB(\gR^m)\otimes\bigwedge\gR^n
\end{equation*}
of complex-valued bounded smooth superfunctions with every derivative bounded. It is a generalization of the space $\caB(\gR^m)$ of Schwartz to the graded setting. Endowed with the seminorms
\begin{equation}
|f|_{\alpha} =\sup_{x\in\gR^m}\{\sum_I|D_x^\alpha f_I(x)|\}\label{eq-seminormB}
\end{equation}
and the pointwise product, this space is a Fr\'echet superalgebra. See e.g. \cite{Inoue:1991,Inoue:2003} for close examples of Fr\'echet superalgebras and related analysis.

The oscillatory integrals give a meaning to expressions like\footnote{we adopt the notation $\dd x_i\dd\xi_i:=\dd x_1\dd\xi_1\dd x_2\dd\xi_2\dots$.} $\int\dd x_i\dd\xi_i\ e^{i\omega_0(x_1,x_2)}f(x_1,\xi_1,x_2,\xi_2)$ for a (non-integrable) function $f\in\caB(\gR^{2m|2n})$. Let us define the operator $O$ by 
\begin{equation*}
(O\fois f)(x_1,\xi_1,x_2,\xi_2)=(1-\Delta_{(x_1,x_2)}) \Big(\frac{1}{1+x_1^2+x_2^2}f(x_1,\xi_1,x_2,\xi_2)\Big),
\end{equation*}
for a smooth superfunction $f$ with compact support and where $\Delta_{(x_1,x_2)}$ denotes the Laplacian with respect to the variables $(x_1,x_2)\in \gR^m\times\gR^m$. An integration by parts shows that
\begin{equation}
\int\dd x_i\dd\xi_i\ e^{i\omega_0(x_1,x_2)}f(x_1,\xi_1,x_2,\xi_2)= \int\dd x_i\dd\xi_i\ e^{i\omega_0(x_1,x_2)}(O^k\fois f)(x_1,\xi_1,x_2,\xi_2),\label{eq-qu-oscil}
\end{equation}
for any $k\in\gN$. Moreover, there exist (bounded) functions $b^\alpha\in\caB(\gR^{2m})$ such that
\begin{equation}
(O^k\fois f)(x_1,\xi_1,x_2,\xi_2)=\frac{1}{(1+x_1^2+x_2^2)^k}\sum_{\alpha\in\gN^{2m},\,|\alpha|\leq 2k} b^\alpha(x_1,x_2) D^\alpha f(x_1,\xi_1,x_2,\xi_2).\label{eq-qu-oscil2}
\end{equation}
As a consequence, for any $f\in\caB(\gR^{2m|2n})$, there exists an integer $k$ such that $(O^k\fois f)\in L^1(\gR^{2m|2n})$. Thus, the oscillatory integral of $f$ is given by the RHS member of \eqref{eq-qu-oscil}. With this notion, the formula \eqref{eq-prod-moy} defines an associative product on $\caB(\gR^{m|n})$.

\subsection{Universal deformation formula}
\label{subsec-udf}

In this subsection, we consider an action of the supergroup $\gR^{m|n}$ on a Fr\'echet algebra $(\algA,|\fois|_j)$:
\begin{equation*}
\rho:\gR^{m|n}\times(\algA\otimes\superA)\to(\algA\otimes\superA),
\end{equation*}
satisfying the conditions:
\begin{itemize}
\item $\rho_0=\text{id}$; $\forall z_1,z_2\in\gR^{m|n}$, $\rho_{z_1+z_2}=\rho_{z_1}\rho_{z_2}$.
\item $\forall z\in\gR^{m|n}$, $\rho_z:(\algA\otimes\superA)\to(\algA\otimes\superA)$ is an $\superA$-linear automorphism of algebras.
\item By writing $z=(x,\xi)\in\gR^{m|n}$, we can expand the action as: $\rho_{(x,\xi)}(a)=\sum_I \rho_x(a)_I\xi^I$; $\forall a\in\algA$, $\forall I$, $x\mapsto \rho_x(a)_I$ is $\algA$-valued and continuous.
\item There exists a constant $C>0$ such that
\begin{equation*}
\forall a\in\algA,\, \forall I,\,\forall j,\quad \exists k,\quad \forall x\in\gB M, \quad |\rho_x(a)_I|_j\leq C|a|_k.
\end{equation*}
\end{itemize}
We notice that the star-product \eqref{eq-prod-moy} can be trivially extended to $\algA$-valued superfunctions $\caB_\algA(\gR^{m|n})$ that are bounded with every derivative bounded. Note that this space is also Fr\'echet for the seminorms $|f|_{j,\alpha} =\sup_{x\in\gR^m}\{\sum_I|D^\alpha f_I(x)|_j\}$.

With the action $\rho$, we can deform the product of $\algA$ by this extended star-product.
\begin{definition}[smooth vectors]
The set of smooth vectors of $\algA$ for the action $\rho$ is defined as
\begin{equation*}
\algA^\infty=\{a\in\algA,\quad \rho^a:=z\mapsto \rho_z(a)\text{ is smooth on }\gR^{m|n}\}.
\end{equation*}
\end{definition}

We recall the following Lemma proved in \cite{Bieliavsky:2010su}.
\begin{lemma}
The set of smooth vectors $\algA^\infty$ is dense in $\algA$. Moreover, for any $a\in\algA^\infty$, the map $\rho^a$ lies in $\caB_{\algA^\infty}(\gR^{m|n})$.
\end{lemma}
This means that we can now form the star-product of $\rho^a$ and $\rho^b$, for $a$ and $b$ smooth vectors.

\begin{proposition}[\cite{Bieliavsky:2010su}]
\label{prop-udfdirect}
The expression $a\star_\rho b:=(\rho^a\star\rho^b)(0)$, for $a,b\in\algA^\infty$, yields an associative product on $\algA^\infty$. Endowed with the seminorms
\begin{equation*}
|a|_{j,\alpha}:=|\rho^a|_{j,\alpha}=\sup_{x\in\gR^m}\{\sum_I|D^\alpha \rho_x(a)_I|_j\},
\end{equation*}
$(\algA^\infty,\star_\rho)$ is a (noncommutative) Fr\'echet algebra.
\end{proposition}

It turns out that the star-product \eqref{eq-prod-moy} can be rewritten as
\begin{equation*}
(f_1\star f_2)(x,\xi)=\kappa\int\dd x_1\dd\xi_1\dd x_2\dd\xi_2\ f_1(x_1+x,\xi_1+\xi)f_2(x_2+x,\xi_2+\xi)
e^{\frac{-2i}{\theta}(\omega_0(x_1,x_2)+2\xi_1\xi_2)}.
\end{equation*}
Then, we can write directly the twist $F:\algA^\infty\otimes\algA^\infty\to\algA^\infty\otimes\algA^\infty$ associated to the deformation
\begin{equation}
F=\kappa\int_{\gR^{m|n}\times\gR^{m|n}}\dd z_1\dd z_2\ e^{-\frac{2i}{\theta}\omega(z_1,z_2)}\rho_{z_1}\otimes\rho_{z_2},\label{eq-udf-twist}
\end{equation}
with $z=(x,\xi)\in\gR^{m|n}$ and where $\rho$ replaces the translation for a general action $\rho$ on an algebra $\algA$. Denoting by $\mu_0:\algA\otimes\algA\to\algA$ the undeformed product of $\algA$, we can express the deformed product of Proposition \eqref{prop-udfdirect} as $\mu_F:=\mu_0\circ F$, namely, $\mu_F(a\otimes b)= a\star_\rho b$. The expression \eqref{eq-udf-twist} is also called the universal deformation formula of the supergroup $\gR^{m|n}$ as it can deform a dense subspace $\algA^\infty$ of every algebra $\algA$ on which $\gR^{m|n}$ acts (with some regularity assumed at the beginning of this section).

We can now show new properties regarding the twist of this deformation. Let us recall the definition of the projective tensor product \cite{Grothendieck:1966} of two Fr\'echet algebras $(\algA,|\fois|_j)$ and $(\algB,|\fois|_k)$. It is the completion of the algebraic tensor product $\algA\otimes\algB$ for the family of seminorms: $\forall c\in\algA\otimes\algB$,
\begin{equation}
\pi_{j,k}(c)=\inf\Big\{\sum_i |a_i|_j|b_i|_k,\quad c=\sum_i a_i\otimes b_i\Big\},\label{eq-seminormpi}
\end{equation}
where the infimum is taken over all decompositions $c=\sum_i a_i\otimes b_i$. This completion is denoted by $\algA\widehat\otimes_{\pi}\algB$.

\begin{proposition}
The twist $F$ is a continuous endomorphism on the projective tensor product of $\algA^\infty$ with itself: $F\in\caL(\algA^\infty\widehat{\otimes}_\pi\algA^\infty)$.
\end{proposition}
\begin{proof}
Let $c\in\algA^\infty\otimes\algA^\infty$. Then,
\begin{equation*}
\pi_{j,\alpha;k,\beta}(F(c))=\inf\{|\sum_i F(a_i\otimes b_i)|_{j,\alpha,k,\beta}\}
\end{equation*}
where $c$ can be written as $\sum_i a_i\otimes b_i$, and the infimum is taken over all such decompositions. By using the definition of oscillatory integral \eqref{eq-qu-oscil}, and defining the partial operators
\begin{equation}
(O_{z_1}\fois f)(z_1,z_2)=\frac{1}{1+x_1^2}(1-\Delta_{x_2}) f(z_1,z_2),\quad (O_{z_2}\fois f)(z_1,z_2)=\frac{1}{1+x_2^2}(1-\Delta_{x_1}) f(z_1,z_2),\label{eq-oposcil}
\end{equation}
with $z_i=(x_i,\xi_i)\in\gR^{m|n}$, we obtain
\begin{multline*}
\pi_{j,\alpha;k,\beta}(F(c))=\inf\Big|\kappa'\int\dd z_1\dd z_2 e^{-\frac{2i}{\theta}\omega(z_1,z_2)}O^{k_1}_{z_1}O^{k_2}_{z_2}\sum_i  \rho_{z_1}(a_i)\otimes \rho_{z_2}(b_i)\Big|_{j,\alpha,k,\beta}\\
\leq \inf|\kappa'|\sum_{i,I,J}\int\dd x_1\dd x_2 \frac{1}{(1+x_1^2)^{k_1}(1+x_2^2)^{k_2}}
\quad\sum_{\gamma,\delta} |b^\gamma_1(x_1)b^\delta_2(x_2)| |D^\gamma\rho_{x_1}(a_i)_I|_{j,\alpha} | D^\delta\rho_{x_2}(b_i)_J|_{k,\beta}
\end{multline*}
in the notation of \eqref{eq-qu-oscil2}, if $I,J$ are summed over $\{1,\dots,n\}$ with some conditions, and for $\kappa'$ a constant. By definition of the seminorm,
\begin{equation*}
|D^\gamma\rho_{x_1}(a_i)_I|_{j,\alpha}=\sup_{x_3\in\gR^m}\{\sum_K|D^\alpha_{x_3} \rho_{x_3}(D^\gamma_{x_1}\rho_{x_1}(a_i)_I)_K|_j\}.
\end{equation*}
Since $\rho$ is a group action, we can deduce that
\begin{equation}
\rho_{x_3}(D^\gamma_{x_1}\rho_{x_1}(a_i)_I)_K=(-1)^{|I||K|}\eps(I,K) D^\gamma_{x_1} \rho_{x_1+x_3}(a_i)_{I\cup K}.\label{eq-udf-gr}
\end{equation}
We then choose sufficiently large numbers $k_1$ and $k_2$ such that there exists a constant $C>0$ with
\begin{equation*}
\pi_{j,\alpha;k,\beta}(F(c))\leq C\inf\sum_{i,\gamma,\delta}|a_i|_{j,\alpha+\gamma}|b_i|_{k,\beta+\delta}= C \sum_{\gamma,\delta}|c|_{j,\alpha+\gamma,k,\beta+\delta},
\end{equation*}
where the sum on multi-indices $\gamma,\delta\in\gN^m$ satisfies the constraint $|\gamma|\leq 2k_1$ and $|\delta|\leq 2 k_2$. The last inequality shows that $F$ is continuous on $\algA^\infty\widehat{\otimes}_\pi\algA^\infty$.
\end{proof}

\begin{example}
If we take $\algA=\caB(\gR^{m|n})$ and $\rho_z(f)(z')=f(z+z')$, then the space of smooth vectors is $\algA^\infty=\caB(\gR^{m|n})$ and the product $\mu_F$ corresponds to \eqref{eq-prod-moy}.
\end{example}
There are a lot of other examples, like the actions of $\gR^{m|n}$ over a certain class of continuous superfunctions on the trivial supermanifolds on which $\gR^{m|n}$ is acting (see \cite{Bieliavsky:2010su}).

\subsection{External symmetries of the deformation}
\label{subsec-hopf}

To introduce the external symmetries of the deformation or of the twist $F$, we need the notion of topological Hopf algebra, endowed with a Fr\'echet topology.

\begin{definition}
\label{def-frhopf}
A Fr\'echet-Hopf algebra is a Hopf algebra $H$ endowed with a Fr\'echet topology, such that the algebraic operations - product, unit, coproduct, counit and antipode - are continuous maps for the Fr\'echet structure and for a given topological tensor product on $H$.

Given a Fr\'echet-Hopf algebra $H$ with topological tensor product $\widehat\otimes_{HH}$, as well as a topological tensor product $\widehat\otimes_{\algA H}$ between $H$ and a Fr\'echet algebra $\algA$ that has itself another topological tensor product $\widehat\otimes_{\algA\algA}$; we say that $\algA$ is a comodule-algebra of $H$ if it is an algebraic comodule-algebra of $H$, if the coaction can be continuously extended to
\begin{equation*}
\algA\to\algA\widehat\otimes_{\algA H} H
\end{equation*}
and if the three topological tensor are compatible, i.e. if the flips involved in the axioms of a comodule-algebra are continuous for the Fr\'echet structures (see Lemma \ref{lem-com-exchange} for an example).
\end{definition}

In the context of superspaces, we can introduce the following Fr\'echet-Hopf algebra. Let $H:=\caB(\gR^{m|n})$ with its Fr\'echet topology \eqref{eq-seminormB}. We introduce a topological tensor product different from the projective one, denoted by $\tau$, as follows. We define $\algA\widehat\otimes_{\tau} H$ to be the completion of the algebraic tensor product for the family of seminorms of $\caB_\algA(\gR^{m|n})$:
\begin{equation}
\tau_{j,\alpha}(f)=|f|_{j,\alpha} =\sup_{x\in\gR^m}\Big\{\sum_I|D^\alpha f_I(x)|_j\Big\}.\label{eq-seminormtau}
\end{equation}
One can then see that $H\widehat\otimes_{\tau}H\simeq \caB(\gR^{m|n}\times \gR^{m|n})$ and by definition, $\algA\widehat\otimes_{\tau}H\simeq \caB_\algA(\gR^{m|n})$. On $H$ we consider the standard Hopf algebra structure, whose algebraic operations can be continuously extended for the tensor product $\tau$:
\begin{itemize}
\item the product $\mu:H\widehat\otimes_\tau H\to H$ defined by $\mu(f_1\otimes f_2)(z)=f_1(z) f_2(z)$,
\item the unit $\gone:\gC\to H$ defined by $\gone(\lambda)(z)= \lambda$,
\item the coproduct $\Delta:H\to H\widehat\otimes_\tau H$ defined by $\Delta f( z_1,z_2)=f(z_1 z_2)$,
\item the counit $\eps: H\to\gC$ defined by $\eps(f)=f(0)$,
\item the antipode $S:H\to H$ defined by $S f(z)=f(-z)$,
\end{itemize}
where $f_i\in H$, $z_i\in \gR^{m|n}$, $\lambda\in\gC$. These operations satisfy the useful axioms of Hopf algebra, taking into account that the flip $\sigma_{12}:H\otimes H\to H\otimes H$ is defined by
\begin{equation}
\sigma_{12}(f_1\otimes f_2)=(-1)^{|f_1||f_2|}f_2\otimes f_1\label{eq-flip}
\end{equation}
because of the grading. This means that for $f\in H\widehat\otimes_\tau H$, $\sigma_{12}f(z,z')=f(z',z)$.

\begin{proposition}
\label{prop-hopf-smooth}
$H=\caB(\gR^{m|n})$ is a $\gZ_2$-graded supercommutative Fr\'echet-Hopf algebra for the topological tensor product $\tau$.
\end{proposition}
\begin{proof}
Due to the explicit expression of the coproduct
\begin{equation*}
\Delta(f)(x_1,\xi_1;x_2,\xi_2)=\sum_{I,J}\eps(I,J)f_{I\cup J}(x_1+x_2)\xi_1^I\xi_2^J
\end{equation*}
obtained by an expansion on the odd variables and by \eqref{eq-eps}, we have $\forall f\in\caB(\gR^{m|n})$,
\begin{align*}
\tau_{\alpha,\beta}(\Delta(f))=\sup_{x_1,x_2\in\gR^m}\{\sum_{I,J}|\eps(I,J)D^\alpha_{x_1} D^\beta_{x_2} f_{I\cup J}(x_1+x_2)|\}\leq 2^n |f|_{\alpha+\beta},
\end{align*}
which shows the continuity of $\Delta:H\to H\widehat\otimes_\tau H$. The continuity of the other operations can be proved in the same way. The algebraic properties between operations are the same as in the non-graded setting except $(S\otimes S)\Delta=\sigma_{12}\Delta S$ involving the flip \eqref{eq-flip}. It can be showed that
\begin{equation*}
\sigma_{12}\Delta(f)(x_1,\xi_1;x_2,\xi_2)=\Delta(f)(x_2,\xi_2;x_1,\xi_1)=\Delta(f)(x_1,\xi_1;x_2,\xi_2),
\end{equation*}
because $\gR^{m|n}$ is Abelian. Then, we have
\begin{equation*}
(S\otimes S)\Delta(f)(x_1,\xi_1;x_2,\xi_2)=f(-x_1-x_2,-\xi_1-\xi_2)=\sigma_{12}\Delta S(f)(x_1,\xi_1;x_2,\xi_2).
\end{equation*}
\end{proof}

\begin{remark}
\label{rmk-nuclear}
Note that $\caC^\infty(\gR^{m|n})$ is also a Fr\'echet-Hopf algebra (see \cite{Bonneau:2003vb} in the non-graded setting). Since it is nuclear contrary to $\caB(\gR^{m|n})$, this structure is independent of the choice of the topological tensor product. In this paper, we consider $\caB(\gR^{m|n})$ for the deformation quantization since $\caC^\infty(\gR^{m|n})$ is too large for the star-product to be defined on it (see section \ref{subsec-prod}). $\caB(\gR^{m|n})$ is not nuclear but we will see that the tensor products $\tau$ and $\pi$ (needed for representations) are compatible in a certain sense. We could of course have considered a smaller nuclear subalgebra like the Schwartz algebra $\caS(\gR^{m|n})$ - see \cite{Bieliavsky:2010ab} in the non-graded setting - but then the coproduct does not stabilize this algebra and we have to see it as valued in (the tensor product of) the multiplier algebra of $\caS(\gR^{m|n})$. See also \cite{Voigt:2008} for another framework (bornological vector spaces) adapted to quantum groups.
\end{remark}

Let us present the dual version of the universal deformation formula studied in section \ref{subsec-udf}, which will lead to the external symmetries. As before, we consider the Fr\'echet-Hopf algebra $H=\caB(\gR^{m|n})$ associated to the supergroup $\gR^{m|n}$. The reformulation of the action $\rho$ in this context will be done by the notion of $H$-comodule algebras (see Definition \ref{def-frhopf}). To this aim, we need the following intermediate result.
\begin{lemma}
\label{lem-com-exchange}
The topological tensor product $\tau$ is compatible with the projective one $\pi$, in the sense that the flip
\begin{equation*}
\sigma_{23}:(\algA\widehat\otimes_\tau H)\widehat\otimes_\pi(\algA\widehat\otimes_\tau H)\to (\algA\widehat\otimes_\pi\algA)\widehat\otimes_\tau(H\widehat\otimes_\tau H),
\end{equation*}
defined by $\sigma_{23}(a_1\otimes f_1\otimes a_2\otimes f_2)=a_1\otimes a_2\otimes f_1\otimes f_2$, is continuous, for any Fr\'echet algebra $(\algA,|\fois|_j)$.
\end{lemma}
\begin{proof}
For $a_i,b_i\in\algA$ and $f_i,g_i\in H$, due to the expressions \eqref{eq-seminormpi} and \eqref{eq-seminormtau} of the seminorms of $\pi$ and $\tau$, one has
\begin{align*}
\pi_{j,\alpha;k,\beta}(\sum_i a_i\otimes f_i\otimes b_i\otimes g_i)&=\inf \sum_i \tau_{j,\alpha}(a_i\otimes f_i)\tau_{k,\beta}(b_i\otimes g_i)\\
&=\inf \sum_i \sup_{x,y\in\gR^m}\sum_{I,J} |a_i|_j|D^\alpha f_{i,I}(x)||b_i|_k|D^\beta g_{i,J}(y)|.
\end{align*}
Moreover,
\begin{align*}
\tau_{j,k;\alpha,\beta}(\sigma_{23}(\sum_i a_i\otimes f_i\otimes b_i\otimes g_i)) &= \sup_{x,y\in\gR^m}\sum_{I,J}\pi_{j,k}(\sum_i (a_i\otimes b_i) D^\alpha f_{i,I}(x)D^\beta g_{i,J}(y))\\
&=\sup_{x,y\in\gR^m}\sum_{I,J}\inf \sum_i|a_i|_j|b_i|_k|D^\alpha f_{i,I}(x)||D^\beta g_{i,J}(y)|.
\end{align*}
Since $\forall x,y\in\gR^m$,
\begin{equation*}
\inf \sum_i|a_i|_j|b_i|_k|D^\alpha f_{i,I}(x)||D^\beta g_{i,J}(y)|\leq \inf \sum_i\sup_{x,y\in\gR^m}|a_i|_j|b_i|_k|D^\alpha f_{i,I}(x)||D^\beta g_{i,J}(y)|,
\end{equation*}
there exists a constant $1\leq C\leq 2^{n+1}$ such that
\begin{equation*}
\tau_{j,k;\alpha,\beta}(\sigma_{23}(\sum_i a_i\otimes f_i\otimes b_i\otimes g_i))\leq C\, \pi_{j,\alpha;k,\beta}(\sum_i a_i\otimes f_i\otimes b_i\otimes g_i),
\end{equation*}
which proves the continuity of $\sigma_{23}$.
\end{proof}

\begin{proposition}
The action $\rho$ of $\gR^{m|n}$ on a Fr\'echet algebra $(\algA,\mu_0)$ with axioms of section \ref{subsec-udf}, generates the continuous coaction $\chi:\algA^\infty\to\algA^\infty\widehat{\otimes}_\tau H$ defined by
\begin{equation*}
\forall a\in\algA^\infty,\,\forall z\in\gR^{m|n}\quad:\quad \chi(a)(z):=\rho_z(a).
\end{equation*}
Then $(\algA^\infty,\mu_0)$ is an $H$-comodule algebra, with $\widehat\otimes_{\algA H}:=\widehat\otimes_\tau$ and $\widehat\otimes_{\algA\algA}:=\widehat\otimes_\pi$.
\end{proposition}
\begin{proof}
Since $\rho$ is a group action and that $\forall z\in\gR^{m|n}$, $\rho_z:(\algA\otimes\superA)\to(\algA\otimes\superA)$ is an algebra-morphism, we deduce that $\chi$ satisfies the axioms of a coaction:
\begin{equation*}
(\text{id}\otimes\Delta)\chi=(\chi\otimes\text{id})\chi,\quad (\text{id}\otimes\eps)\chi=\text{id}.
\end{equation*}
Thus, $\algA^\infty$ is an algebraic $H$-comodule algebra:
\begin{equation}
(\mu_0\otimes\mu)\sigma_{23}(\chi\otimes\chi)=\chi\mu_0,\label{eq-com-axiom}
\end{equation}
where $\mu_0:\algA^\infty\widehat{\otimes}_\pi\algA^\infty\to\algA^\infty$ corresponds to the undeformed product of $\algA$ and $\sigma_{23}$ is the flip of Lemma \ref{lem-com-exchange} for the algebra $\algA^\infty$. Let $a$ be in $\algA^\infty$; we then have $\chi(a)\in\algA^\infty\widehat{\otimes}_\tau\caB(\gR^{m|n})\simeq \caB_{\algA^\infty}(\gR^{m|n})$, so
\begin{equation*}
\tau_{j,\alpha;\beta}(\chi(a))=\sup_{y\in\gR^m}\sum_I|D^\beta\rho_y(a)_I|_{j,\alpha}
= \sup_{y,y'}\sum_{I,J}|D^\alpha_{y'}\rho_{y'}(D^\beta_y \rho_y(a)_I)_J|_j.
\end{equation*}
By using \eqref{eq-udf-gr}, we obtain
\begin{equation*}
\tau_{j,\alpha;\beta}(\chi(a))=\sup_{y,y'}\sum_{I,J}|\eps(I,J)D^\alpha_{y'}D^\beta_y \rho_{y+y'}(a)_{I\cup J}|_j,
\end{equation*}
which shows that there exists $C>0$ such that $\tau_{j,\alpha;\beta}(\chi(a))\leq |a|_{j,\alpha+\beta}$, i.e. $\chi$ is continuous. Note that the flip $\sigma_{23}$ is continuous due to the compatibility of the topological tensor products $\tau$ and $\pi$ showed in Lemma \ref{lem-com-exchange}. Indeed, all the maps involved in Equation \eqref{eq-com-axiom} have to be continuous in order for $\algA^\infty$ to be a comodule-algebra of the Fr\'echet-Hopf algebra $H$.
\end{proof}

Now, the algebra $(\algA^\infty,\mu_0)$ can be deformed by the twist $F$ defined in \eqref{eq-udf-twist} in such a way $(\algA^\infty,\mu_F=\mu_0 F)$ is a Fr\'echet algebra. The universal deformation formula constructed before provides therefore a deformation of the category of the $H$-comodule algebras. Of course, once deformed, there is a priori no reason for $(\algA^\infty,\mu_F)$ to be again an $H$-comodule algebra.
\begin{definition}
\label{def-extsym}
Given a twist $F$ which deforms the category of comodule-algebras $(\algA,\mu_0)$ of a given Fr\'echet-Hopf algebra $H$, we call external symmetries of the twist the Fr\'echet-Hopf algebra $H_F$ for which any deformed algebra $(\algA,\mu_F)$ is an $H_F$-comodule-algebras.
\end{definition}
In the non-graded setting and formally in the deformation parameter, there is a way to obtain the external symmetries $H_F$ from $H$ and the twist $F$ \cite{Drinfeld:1989,Giaquinto:1998}. This has been extended to non-formal deformations of a large class of solvable Lie groups in \cite{Bieliavsky:2010ab}. Let us describe this process for such a Lie group $G$ and where $H$ denotes (a closed subclass of) $\caC^\infty(G)$ with its Hopf algebra structure. If the non-formal twist of $G$ on algebras $\algA$, where $G$ acts by $\rho$, has the form
\begin{equation*}
F=\int_{G\times G}\dd x_1\dd x_2\ e^{-\frac{2i}{\theta}S(x_1,x_2)}A(x_1,x_2)\rho_{x_1}\otimes\rho_{x_2},
\end{equation*}
where $S$ and $A$ are the phase and amplitude of the deformation quantization, then we can consider the left $L$ and right $R$ actions of $G$ on itself to obtain maps $\caC^\infty(G)\hat\otimes \caC^\infty(G)\to \caC^\infty(G)\hat\otimes \caC^\infty(G)$:
\begin{align}
&F_L=\int_{G\times G}\dd x_1\dd x_2\ e^{-\frac{2i}{\theta}S(x_1,x_2)}A(x_1,x_2)R^*_{x_1}\otimes R^*_{x_2},\nonumber\\
&F_R=\int_{G\times G}\dd x_1\dd x_2\ e^{-\frac{2i}{\theta}S(x_1,x_2)}A(x_1,x_2)L^*_{(x_1)^{-1}}\otimes L_{(x_2)^{-1}}.\label{eq-extsymsolv}
\end{align}
To obtain the external symmetries of $F$, the product $\mu$ of $H$ has to be twisted \cite{Bieliavsky:2010ab} into: $\mu_{QG}:= F_L\circ F_R\circ\mu$, which is compatible with the undeformed coproduct $\Delta$. Thus, (a subclass of) $\caC^\infty(G)$ with $\mu_{QG}$ and $\Delta$ is the topological Hopf algebra corresponding to the external symmetries.

In the graded setting, the construction has not been provided in general. However, for the supergroup $\gR^{m|n}$, we can see that the external symmetries of the deformation of $\gR^{m|n}$ are $H=(\caB(\gR^{m|n}),\mu,\Delta)$ without twisting its product.
\begin{proposition}
\label{prop-extsym}
$(\algA^\infty,\mu_F)$ is an $H$-comodule algebra.
\end{proposition}
\begin{proof}
The only remaining condition to check is $(\mu_F\otimes\mu)\sigma_{23}(\chi\otimes\chi)=\chi\mu_F$. For $a,b\in\algA^\infty$ and $z\in\gR^{m|n}$,
\begin{align*}
&\chi\mu_F(a\otimes b)(z)=\kappa\int\dd z_1\dd z_2 e^{-\frac{2i}{\theta}\omega(z_1,z_2)}\rho_z(\rho_{z_1}(a)\rho_{z_2}(b))\\
&(\mu_F\otimes\mu)\sigma_{23}(\chi\otimes\chi)(a\otimes b)(z)= \kappa\int\dd z_1\dd z_2 e^{-\frac{2i}{\theta}\omega(z_1,z_2)} \rho_{z_1}\rho_z(a)\rho_{z_2}\rho_z(b).
\end{align*}
Since $\rho_z$ is an algebra-morphism, $\rho$ a group action and $\gR^{m|n}$ an Abelian supergroup, we obtain that $\chi\mu_F(a\otimes b)(z)=(\mu_F\otimes\mu)\sigma_{23}(\chi\otimes\chi)(a\otimes b)(z)$.
\end{proof}
Note that $(\caB(\gR^{m|n}),\mu_F=\mu\circ F,\Delta)$ is not a Hopf algebra anymore: the deformed product $\mu_F$ is not compatible with the undeformed coproduct $\Delta$.

\section{Construction of quantum supergroups}

\subsection{Definition of a Fr\'echet quantum supergroup}

In Definition \ref{def-extsym}, we saw that external symmetries of the deformation quantization of actions of a Lie group on Fr\'echet algebras correspond to a deformation of the Fr\'echet-Hopf algebra associated to the Lie group by using \eqref{eq-extsymsolv}. It forms a quantum group.

In the case of $\gR^{m|n}$, we saw in Proposition \ref{prop-extsym} that the external symmetries of the twist $F$ correspond to the group $\gR^{m|n}$ itself (i.e. the undeformed Hopf algebra $H=\caB(\gR^{m|n}$)), because $\gR^{m|n}$ is Abelian. However, to anticipate what could be the external symmetries of a more general supergroup, we have to introduce the new notion of quantum supergroup. Taking into account the nature of external symmetries, we see that this notion has to correspond to a topological graded Hopf algebra, but is not supercommutative in general.
\begin{definition}
\label{def-qusuper}
We define a Fr\'echet quantum supergroup to be a Fr\'echet-Hopf algebra (see Definition \ref{def-frhopf}), for a given topological tensor product, with a $\gZ_2$-grading and for which the algebraic operations - product, unit, coproduct, counit and antipode - respect this grading, i.e. are homogeneous maps of degree 0.
\end{definition}

There exist in the literature other definitions of quantum supergroups, as there are different notions of quantum groups - related to topological Hopf algebras or using deformations of universal enveloping algebras of Lie algebras. In particular, the purely algebraic version of Definition \ref{def-qusuper} corresponds exactly to the notion of quantum supergroup in \cite{Majid:1995}. But here, we place ourselves in the context of topological Hopf (super)-algebras. Note also that we do not assume that the Fr\'echet-Hopf algebra has to be nuclear (see Remark \ref{rmk-nuclear}).

\begin{remark}
In the case of $\gR^{m|n}$, the definition of a supergroup given by Kostant \cite{Kostant:1977} is equivalent to the data of the sheaf $\caC^\infty$ or $\caB$ assuming that $\caC^\infty(\gR^{m|n})$ or $\caB(\gR^{m|n})$ is a $\gZ_2$-graded commutative Fr\'echet-Hopf algebra. We can notice indeed that the conditions in \cite{Kostant:1977} of smoothness on the coproduct and the antipode are equivalent to continuity conditions for the Fr\'echet structure. This is why Definition \ref{def-qusuper} is an extension of Kostant's definition of a supergroup to the quantum level, omitting the supercommutativity condition.
\end{remark}

Following again the analogy with external symmetries of the deformation quantization of $\gR^{m|n}$, we introduce the representations of a Fr\'echet quantum supergroup.
\begin{definition}
A representation of a given Fr\'echet quantum supergroup $H$ is a $\gZ_2$-graded comodule-algebra $\algA$ of $H$ (see Definition \ref{def-frhopf}) such that the continuous coaction $\algA\to \algA\widehat\otimes_{\algA H} H$ is homogeneous of degree 0.
\end{definition}

\subsection{The Clifford algebra}

In this section, we consider the simplest example of Clifford algebra, for which we present the structure of (Fr\'echet) quantum supergroup. The Clifford algebra can be seen as a deformation quantization of the superspace $\gR^{0|n}$: $\forall f_1,f_2\in\caC^\infty(\gR^{0|n})$,
\begin{equation*}
(f_1\star f_2)(\xi)=\kappa\int \dd\xi_1\dd \xi_2\ f_1(\xi_1+\xi)f_2(\xi_2+\xi) e^{-\frac{4i}{\theta}\xi_1\xi_2},
\end{equation*}
for $\xi\in\gR^{0|n}$. The above star-product corresponds actually to \eqref{eq-prod-moy} for $m=0$. We set $H:=\caC^\infty(\gR^{0|n})$ and we can endow it with the norm $\norm f\norm:=\sum_I | f_I|$ for $I$ to be summed over the parts of $\{1,\dots,n\}$. On this finite-dimensional space, any other norm would have been equivalent so that we do not look anymore at the topology of this example. $H$ is associative, with unit $1$. As generators, we take$ e_i:=\xi^i$ with $\xi=(\xi^1,\dots,\xi^n)\in\gR^{0|n}$. Since
\begin{equation*}
 e_i\star e_j= e_ie_j+\frac{i\theta}{4}\delta_{ij},
\end{equation*}
we have the following relations of $Cl(n,\gC)$:
\begin{equation*}
e_i\star e_j+e_j\star e_i=\frac{i\theta}{2}\delta_{ij}.
\end{equation*}
If $\theta=-4i$, we can endow $H$ \cite{Albuquerque:2000qk} with a structure of quantum supergroup:
\begin{itemize}
\item Coproduct: $\Delta(e_i):=e_i\otimes e_i$,
\item Counit: $\eps(e_i):=1$,
\item Antipode: $S(e_i):=e_i$,
\item Product on tensors: $(e_i\otimes e_j)\star (e_k\otimes e_l):=\sigma_{jk} (e_i\star e_k)\otimes (e_j\star e_l)$,
\end{itemize}
with $\sigma_{ij}=1$ if $i\leq j$ and $\sigma_{ij}=-1$ if $i>j$. Note that $\sigma_{ij}$ is a Schur multiplier of the group $\gZ_2^n$ for which the algebra $Cl(n,\gC)$ is $\gZ_2^n$-graded commutative \cite{deGoursac:2008bd}. A corresponding Kac-Takesaki operator would be given by $W(e_i\otimes e_j):=e_i\otimes(e_i\star e_j)$.

\subsection{Examples of solvable Fr\'echet quantum supergroups}

Let us now construct other examples of Fr\'echet quantum supergroups, which are deformation of solvable Lie supergroups. These are consistent extensions of \cite{Rieffel:1992} to the graded setting. We consider a $(1|0)$-dimensional split extension of the symplectic superspace $(\gR^{m|n},\omega)$ of Definition \ref{def-concrete-superspace}. Let indeed $\pi:\gR^{1|0}\to Sp(\gR^{m|n},\omega)$ be a symplectic representation of $\gR^{1|0}$ on $\gR^{m|n}$, homogeneous of degree 0. It can be written as $\pi=\begin{pmatrix} \pi_0 & 0 \\ 0 & \pi_1 \end{pmatrix}$ (square matrix of size $m+n$). We also assume each matrix coefficient of $\pi$ to be smooth with respect to the variable $a\in\gR^{1|0}$. Then, the split extension is of the form $G:=\gR^{1|0}\ltimes_\pi \gR^{m|n}$ with supergroup law:
\begin{equation}
(a,x,\xi)\fois (a',x',\xi')=\left( a+a',\pi_0(a')x+x',\pi_1(a')\xi+\xi'\right).\label{eq-lawsolv}
\end{equation}
Here $a\in\gR^{1|0}$, $x\in\gR^{m|0}$ and $\xi\in\gR^{0|n}$. We use the natural action of $\gR^{m|n}$ on $G$ together with the universal deformation formula of Proposition \ref{prop-udfdirect} to deform the product of functions on $G$ as
\begin{equation}
(f_1\star f_2)(a,x,\xi)=\kappa(a)\int \dd x_1\dd\xi_1\dd x_2\dd \xi_2\ f_1(a,x_1+x,\xi_1+\xi)f_2(a,x_2+x,\xi_2+\xi) e^{-\frac{2i}{a}(\omega_0(x_1,x_2)+2\xi_1\xi_2)}\label{eq-prodsolv}
\end{equation}
with $\kappa(a)=(-1)^{\frac{n(n+1)}{2}}\frac{(ia)^n}{4^n(\pi a)^m}$. Note that we used the extension variable $a$ as the deformation parameter. This will be crucial to define a consistent coproduct. We define $H$ to be the space of smooth superfunctions on $G$ that are bounded with every derivative bounded in the variables $(x,\xi)\in\gR^{m|n}$:
\begin{equation*}
H:=\caC^\infty(\gR^{1|0})\widehat\otimes\caB(\gR^{m|n}).
\end{equation*}
The standard Fr\'echet structure of $H$ is defined by the seminorms
\begin{equation}
|f|_{\alpha,K,\beta} =\sup_{a\in K,\,x\in\gR^m}\{\sum_I|D_a^\alpha D_x^\beta f_I(a,x)|\}\label{eq-seminormsolv}
\end{equation}
for $K$ compact of $\gR=\gB(\gR^{1|0})$, $\alpha\in\gN$ and $\beta\in\gN^m$.

\begin{proposition}
\label{prop-fralgsolv}
Endowed with the star-product \eqref{eq-prodsolv} and the seminorms \eqref{eq-seminormsolv}, $H$ is a unital associative Fr\'echet superalgebra.
\end{proposition}
\begin{proof}
What remains to prove here is the continuity of the star-product \eqref{eq-prodsolv}. Let $f_1,f_2\in H$, $K$ compact of $\gR$, $\alpha\in\gN$ and $\beta\in\gN^m$. First we perform a change of variable: $x_1\mapsto ax_1$ in the expression of $|f_1\star f_2|_{\alpha,K,\beta}$. Then, we can estimate this expression by expanding the superfunctions $f_1$ and $f_2$ along the odd variables in \eqref{eq-prodsolv} and integrate over these odd variables, and also apply operators \eqref{eq-oposcil} inside the integrals. Thus for $k_1,k_2\in\gN$, there exist functions $b_1^\gamma,b_2^\delta\in\caB(\gR^m)$ such that
\begin{multline*}
|f_1\star f_2|_{\alpha,K,\beta}\leq \frac{1}{4^n\pi^m}\sup_{a\in K,\,x\in\gR^m}\sum_{I,J,\gamma,\delta,\tau,\nu,\mu}\int\dd x_1\dd x_2 \frac{1}{(1+x_1^2)^{k_1}(1+x_2^2)^{k_2}} |b_1^\gamma(x_1)b_2^\delta(x_2)|\\
|a|^{|\mu|} |D_a^\tau D^\gamma_x (f_1)_I(a,x+ax_1)| |D_a^\nu D^\delta_x (f_2)_J(a,x+x_2)| 
\end{multline*}
where $I,J$ are summed over $\{1,\dots,n\}$ with some conditions; $\tau,\nu,\mu$ over $\gN$ with $\tau+\nu\leq\alpha$ and $\mu\leq |\gamma|$; $\gamma,\delta$ over $\gN^m$ with $|\gamma|\leq |\beta|+2k_1$ and $|\delta|\leq |\beta|+2k_2$. For an adapted choice of $k_1,k_2$, it means that there exists a constant $C>0$ such that
\begin{equation*}
|f_1\star f_2|_{\alpha,K,\beta}\leq C\sum_{\gamma,\delta,\tau,\nu}|f_1|_{\tau,K,\gamma}|f_2|_{\nu,K,\delta}
\end{equation*}
where the sum is finite. This proves the continuity of the star-product.
\end{proof}

We then consider the coproduct, counit and antipode coming from the (undeformed) supergroup structure of $G$:
\begin{itemize}
\item the coproduct $\Delta:H\to H\widehat\otimes H$ defined by $\Delta f(g,g')=f(g\fois g')$ for $g,g'\in G$ and the supergroup law \eqref{eq-lawsolv},
\item the counit $\eps: H\to\gC$ defined by $\eps(f)=f(0,0,0)$,
\item the antipode $S:H\to H$ defined by $S f(g)=f(g^{-1})$, with $f\in H$ and $(a,x,\xi)^{-1}=(-a,-\pi_0(-a)x,-\pi_1(-a)\xi)$.
\end{itemize}
We note $\mu:H\widehat\otimes H\to H$ the star-product: $\mu(f_1\otimes f_2):=f_1\star f_2$.

\begin{theorem}
\label{thm-solv}
$(H,\mu,1,\Delta,\eps,S)$ is a Fr\'echet quantum supergroup.
\end{theorem}
\begin{proof}
We know from Proposition \ref{prop-fralgsolv} that $(H,\star,1)$ is a Fr\'echet superalgebra. Let us show first that the coproduct is continuous. For $f\in H$, $(a,x,\xi),(a',x',\xi')\in G$, the coproduct takes the form
\begin{equation*}
\Delta(f)(a,x,\xi,a',x',\xi')=\sum_{I,J,L}\eps(I,J) f_{I\cup J}(a+a',\pi_0(a')x+x')(\pi_1(a'))_{IL}\xi^L(\xi')^J
\end{equation*}
with some constraints on $I,J,L$, and $\eps(I,J)$ given by \eqref{eq-eps}. Then,
\begin{align*}
|\Delta(f)|_{\alpha,K,\beta;\alpha',K',\beta'}&\leq \sup_{a\in K,\,a'\in K',\,x,x'\in\gR^m}\sum_{I,J,L} |D_a^\alpha D_{a'}^{\alpha'}D_x^\beta D_{x'}^{\beta'} f_{I\cup J}(a+a',\pi_0(a')x+x') (\pi_1(a'))_{IL}|\\
&\leq C\sum_{\tau,\gamma}|f|_{\tau,K'',\gamma}
\end{align*}
where $K''$ is a compact containing $\{a+a',\ a\in K,\ a'\in K'\}$, $\tau\leq\alpha+\alpha'$, $|\gamma|\leq|\beta|+|\beta'|$, and $C$ a constant depending in particular on the smooth matrix coefficients of $\pi$ and their derivatives. This proves that $\Delta$ is continuous. In the same way, the counit $\eps$ and the antipode $S$ are continuous.

Let us show that $\Delta$ is an algebra morphism for the star-product. For $f_1,f_2\in H$, we have
\begin{multline*}
\Delta(f_1\star f_2)(a,x,\xi,a',x',\xi')=\kappa(a+a')\int \dd x_1\dd\xi_1\dd x_2\dd \xi_2\ f_1\big(a+a',x_1+\pi_0(a')x+x',\xi_1+\pi_1(a')\xi+\xi'\big)\\
f_2\big(a+a',x_2+\pi_0(a')x+x',\xi_2+\pi_1(a')\xi+\xi'\big) e^{-\frac{2i}{a+a'}(\omega_0(x_1,x_2)+2\xi_1\xi_2)}
\end{multline*}
Besides,
\begin{multline*}
\Delta(f_1)\star \Delta(f_2)(a,x,\xi,a',x',\xi')=\kappa(a)\kappa(a')\int \dd x_1\dd\xi_1\dd x_2\dd \xi_2\dd x_1'\dd\xi_1'\dd x_2'\dd \xi_2'\\
 f_1\big(a+a',\pi_0(a')(x_1+x)+x_1'+x',\pi_1(a')(\xi_1+\xi)+\xi_1'+\xi'\big) e^{-\frac{2i}{a}(\omega_0(x_1,x_2)+2\xi_1\xi_2)}\\
f_2\big(a+a',\pi_0(a')(x_2+x)+x_2'+x',\pi_1(a')(\xi_2+\xi)+\xi_2'+\xi'\big)e^{-\frac{2i}{a'}(\omega_0(x'_1,x'_2)+2\xi'_1\xi'_2)}.
\end{multline*}
The sign of the star-product of elements of $H\widehat\otimes H$ coming from the flip \eqref{eq-flip} has been taken into account. We perform the change of variables: $x_i''=x_i'+\pi_0(a')x_i$, $\xi_i''=\xi_i'+\pi_1(a')\xi_i$. By using the identity $\int\dd\xi\ e^{c\xi\xi'}=(-1)^{\frac{n(n-1)}{2}}c^n(\xi')^{\{1,\dots,n\}}$, we can integrate over $x_1,\xi_1$, obtaining
\begin{multline*}
\Delta(f_1)\star \Delta(f_2)(a,x,\xi,a',x',\xi')=(-4i)^n (-1)^{\frac{n(n+1)}{2}}\pi^m\kappa(a)\kappa(a')\int \dd x_2\dd \xi_2\dd x_1''\dd\xi_1''\dd x_2''\dd \xi_2''\\ \delta\Big(\frac{a+a'}{aa'}x_2-\frac{1}{a'}\pi_0(a')^*x_2''\Big) \Big(\frac{a+a'}{aa'}\xi_2-\frac{1}{a'}\pi_1(a')^*\xi_2''\Big)^{\{1,\dots,n\}}
 f_1\big(a+a',x_1''+\pi_0(a')x+x',\xi_1''\pi_1(a')\xi+\xi'\big)\\
f_2\big(a+a',x''_2+\pi_0(a')x+x',\xi_2''+\pi_1(a')\xi)+\xi'\big)e^{-\frac{2i}{a'}(\omega_0(x''_1,x''_2-\pi_0(a')x_2)+2\xi''_1(\xi''_2-\pi_1(a')\xi_2))}.
\end{multline*}
In the previous step, we used the fact that $\pi$ is a symplectic representation, i.e. $\omega(\pi_0(a)x,\pi_0(a)y)=\omega_0(x,y)$ and $(\pi_1(a)\xi)(\pi_1(a)\eta)=\xi\eta$. Moreover we denote $\pi_0(a)^*:=\omega_0^{-1}\pi_0(a)^T\omega_0$ and $\pi_1(a)^*:=\pi_1(a)^T$. If we now perform the Dirac integration on $x_2,\xi_2$, we obtain
\begin{equation*}
\Delta(f_1)\star \Delta(f_2)(a,x,\xi,a',x',\xi')=\Delta(f_1\star f_2)(a,x,\xi,a',x',\xi').
\end{equation*}
All the other algebraic identities are the same as in the undeformed case except $\mu(\text{id}\otimes S)\Delta= 1\otimes\eps=\mu(S\otimes\text{id})\Delta$. For this, we compute
\begin{multline*}
\mu(\text{id}\otimes S)\Delta(f)(a,x,\xi)=\\
\kappa(a)\int \dd x_1\dd\xi_1\dd x_2\dd \xi_2\ f(0,\pi_0(-a)(x_1-x_2),\pi_1(-a)(\xi_1-\xi_2)) e^{-\frac{2i}{a}(\omega_0(x_1,x_2)+2\xi_1\xi_2)}=\eps(f).
\end{multline*}
\end{proof}

We can now exhibit the analog of Kac-Takesaki operator $W: H\widehat\otimes H\to H\widehat\otimes H$ associated to this quantum supergroup, also called multiplicative unitary in the non-graded context. It is defined \cite{Baaj:1993,Woronowicz:1996} as $\forall a,b\in H$,
\begin{equation}
W(a\otimes b):= (\Delta a)\star (1\otimes b)= a_{(1)}\otimes (a_{(2)}\star b),\label{eq-defmultunit}
\end{equation}
in the Sweedler notations of the coproduct. Its explicit expression is given by $\forall f\in H\widehat\otimes H$,
\begin{multline}
W(f)(a,x,\xi,a',x',\xi')=\kappa(a')\int \dd x_1\dd\xi_1\dd x_2\dd \xi_2\ e^{-\frac{2i}{a'}(\omega_0(x_1,x_2)+2\xi_1\xi_2)}\\
 f\big(a+a',x_1+\pi_0(a')x+x',\xi_1+\pi_1(a')\xi+\xi',a',x_2+x',\xi_2+\xi'\big).\label{eq-multunitsolv}
\end{multline}
\begin{proposition}
\label{prop-multunitsolv}
The Kac-Takesaki operator \eqref{eq-multunitsolv} is a continuous operator $W: H\widehat\otimes H\to H\widehat\otimes H$ homogeneous of degree 0, and it satisfies the pentagonal relation
\begin{equation*}
W_{12}W_{13}W_{23}=W_{23}W_{12}.
\end{equation*}
\end{proposition}
\begin{proof}
Indeed, as $W=(\mu\otimes\mu)\sigma_{23}(\Delta\otimes 1\otimes\text{id})$, it is continuous. To prove the pentagonal relation where involved signs are different from the non-graded case, we use Sweedler notations for the coproduct since its coassociativity has been showed in Theorem \ref{thm-solv}. On the left side,
\begin{equation*}
W_{12}W_{13}W_{23}(a\otimes b\otimes c)= (-1)^{|a_{(3)}||b_{(1)}|} a_{(1)}\otimes (a_{(2)}\star b_{(1)})\otimes (a_{(3)}\star b_{(2)}\star c),
\end{equation*}
where the sign appears because of the action of $W_{13}=(\mu\otimes\text{id}\otimes\mu)\sigma_{24}(\Delta\otimes\text{id}\otimes 1\otimes\text{id})$ and with
\begin{equation*}
\sigma_{24}(a_1\otimes a_2\otimes a_3\otimes a_4\otimes a_5)=(-1)^{|a_2|(|a_3|+|a_4|)+|a_3||a_4|} a_1\otimes a_4\otimes a_3\otimes a_2\otimes a_5.
\end{equation*}
On the right side,
\begin{equation*}
W_{23}W_{12}(a\otimes b\otimes c)= (-1)^{|a_{(3)}||b_{(1)}|} a_{(1)}\otimes (a_{(2)}\star b_{(1)})\otimes (a_{(3)}\star b_{(2)}\star c),
\end{equation*}
where we used $\Delta(a\star b)=\Delta(a)\star\Delta(b)=(-1)^{|a_{(2)}||b_{(1)}|}(a_{(1)}\star b_{(1)})\otimes (a_{(2)}\star b_{(2)})$ due to Theorem \ref{thm-solv}.
\end{proof}

\begin{remark}
\label{rmk-superunit}
For the Lebesgue-Berezin measure on $\gR^{m|n}$, we can define a ``natural'' superhermitian (not positive definite) scalar product
\begin{equation*}
\langle f_1,f_2\rangle:=\int \dd x\dd\xi\ \overline{f_1(x,\xi)} f_2(x,\xi)
\end{equation*}
and a hermitian positive definite one $\big(f_1,f_2\big):=\langle f_1,\ast f_2\rangle$ by using the Hodge operation
\begin{equation*}
\ast \sum_I f_I(x)\xi^I:=\sum_I \eps(I,\complement I)f_I(x)\xi^{\complement I}.
\end{equation*}
Taking into account the right-invariant measure
\begin{equation*}
\dd^R(a,x,\xi)=\frac{1}{\text{sdet}(\pi(a))}\dd (a,x,\xi)=\frac{\det \pi_1(a)}{\det\pi_0(a)}\dd a\dd x\dd\xi
\end{equation*}
on $G$, it is a straightforward computation using \eqref{eq-multunitsolv} to check that
\begin{multline*}
\int\dd^R(a,x,\xi)\dd^R(a',x',\xi')\ \overline{W(f_1)(a,x,\xi,a',x',\xi')} W(f_2)(a,x,\xi,a',x',\xi')\\
=\int\dd^R(a,x,\xi)\dd^R(a',x',\xi')\ \overline{f_1(a,x,\xi,a',x',\xi')} f_2(a,x,\xi,a',x',\xi')
\end{multline*}
for $f_1,f_2\in (H\widehat\otimes H)\cap L^2(G\times G)$. This means that the operator $W$ is superunitary for the superhermitian scalar product associated to $L^2(G\times G,\dd^Rg\otimes\dd^Rg')$
\begin{equation*}
\langle W(f_1),W(f_2)\rangle=\langle f_1,f_2\rangle,
\end{equation*}
which is not true for the positive definite scalar product $\big(\fois,\fois\big)$. $W$ is a ``multiplicative superunitary'' rather than a multiplicative unitary.
\end{remark}

\subsection{Fr\'echet quantum supergroups with supertoral subgroups}

In this section, we will follow the philosophy of \cite{Rieffel:1993qg} to construct deformation of compact Lie supergroups with supertoral subgroups. Let $G$ be a compact connected Lie supergroup (i.e. its body $\gB G$ is compact connected) with $\Gamma:=\gT^{m|n}$ a supertoral subgroup of $G$. We assume that $m$ is even so that the symplectic superspace $(\gR^{m|n},\omega)$ (see Definition \ref{def-concrete-superspace}) is the Lie algebra of $\Gamma$. We note
\begin{equation*}
z=(x,\xi)\in\gR^{m|n}\quad\mapsto \quad e^{z}=e^{(x,\xi)}\in G
\end{equation*}
the exponential restricted to this Lie algebra. Note that $\caC^\infty(G)\simeq \caC^\infty(\gB G)\otimes\bigwedge\gR^n$ is a Fr\'echet superalgebra for the supercommutative pointwise product (see below Definition \ref{def-concrete-superman}) and the seminorms
\begin{equation}
|f|_{\alpha,K} =\sup_{g\in K,\,|\nu|\leq\alpha}\{\sum_I|D^\nu f_I(g)|\},\label{eq-seminormsubtore}
\end{equation}
for $K$ compact subset of a coordinate chart of $\gB G$, $\alpha\in\gN$ and $D^\nu$ a multi-derivation of order $|\nu|$ for even coordinates. The action $\rho:\gR^{m|n}\times\gR^{m|n}\times\caC^\infty(G)\to \caC^\infty(G)$, defined by
\begin{equation*}
\forall z,z'\in\gR^{m|n},\quad \forall g\in G\quad:\quad\rho_{(z,z')}f(g):=f\big(e^{-z}ge^{z'}\big),
\end{equation*}
allows to deform the pointwise product into the star-product
\begin{equation}
(f_1\star f_2)(g)=\kappa^2\int \dd z_1\dd z_3\dd z_2\dd z_4\ f_1(e^{-z_1}g e^{z_3})f_2(e^{-z_2}ge^{z_4}) e^{-\frac{2i}{\theta}(\omega(z_1,z_2)-\omega(z_3,z_4))},\label{eq-prodsubtore}
\end{equation}
with $\kappa=(-1)^{\frac{n(n+1)}{2}}\frac{(i\theta)^n}{4^n(\pi \theta)^m}$, for $g\in G$ and for any $f_1,f_2\in\caC^\infty(G)$. Note that the underlying symplectic space is $(\gR^{m|n},\omega)\oplus(\gR^{m|n},-\omega)$ where the minus sign, which can also be found in the phase of the star-product, will be crucial. We note $H:=\caC^\infty(G)$.

\begin{proposition}
\label{prop-fralgsubtore}
Endowed with the star-product \eqref{eq-prodsubtore} and the seminorms \eqref{eq-seminormsubtore}, $H$ is a unital associative Fr\'echet superalgebra.
\end{proposition}
\begin{proof}
Associativity is a consequence of the universal deformation formula. Let us check that the star-product is continuous. Then, we use the same method as in the proof of Proposition \ref{prop-fralgsolv} and we get that for $k_i\in\gN$, there exist functions $b_i^{\gamma_i}\in\caB(\gR^m)$ and a constant $C>0$ (depending on $\theta$) such that
\begin{multline*}
|f_1\star f_2|_{\alpha,K}\leq C\sup_{g\in \gB K}\sum_{I,J,\gamma_i,\nu_i}\int\dd x_1\dd x_3\dd x_2\dd x_4 \frac{|c_{\nu_1,\nu_2}|}{(1+x_1^2)^{k_1}(1+x_2^2)^{k_2}(1+x_3^2)^{k_3}(1+x_4^2)^{k_4}}\\
 |b_1^{\gamma_1}(x_1)b_2^{\gamma_2}(x_2)b_3^{\gamma_3}(x_3)b_4^{\gamma_4}(x_4)|
 |D_g^{\nu_1} D_{x_1}^{\gamma_1} D_{x_3}^{\gamma_3} ((f_1)_I(e^{-\theta x_1}g e^{x_3}))| |D_g^{\nu_2} D^{\gamma_2}_{x_2}D^{\gamma_4}_{x_4} ((f_2)_J(e^{-\theta x_2}g e^{x_4}))| 
\end{multline*}
where $I,J$ are summed over $\{1,\dots,n\}$ with some conditions; and $\nu_i,\gamma_i$  are such that $\nu_1+\nu_2\leq\alpha$ and $|\gamma_i|\leq 2k_i$. It follows that there exists a constant $C'>0$ and a compact $K'$ of $\gB G$ containing $\{\gB(e^{-z_1}ge^{z_2}),\ g\in K,\ z_i\in\Gamma\}$ such that
\begin{equation*}
|f_1\star f_2|_{\alpha,K}\leq C'\sum_{\tau,\nu}|f_1|_{\tau,K'}|f_2|_{\nu,K'}
\end{equation*}
where the sum is finite. Therefore, the star-product is continuous.
\end{proof}

Let us endow $H$ with the following (undeformed) operations:
\begin{itemize}
\item the coproduct $\Delta:H\to H\widehat\otimes H$ defined by $\Delta f(g,g')=f(g\fois g')$ for $g,g'\in G$,
\item the counit $\eps: H\to\gC$ defined by $\eps(f)=f(e_G)$, with $e_G$ the neutral element of $G$,
\item the antipode $S:H\to H$ defined by $S f(g)=f(g^{-1})$, with $f\in H$.
\end{itemize}
We denote by $\mu:H\widehat\otimes H\to H$ the star-product: $\mu(f_1\otimes f_2):=f_1\star f_2$.

\vspace*{0.3cm}

\begin{theorem}
\label{thm-subtore}
$(H,\mu,1,\Delta,\eps,S)$ is a Fr\'echet quantum supergroup.
\end{theorem}
\begin{proof}
First, we check the compatibility between the coproduct and the product. Set $f_1,f_2\in H$.
\begin{equation*}
\Delta(f_1\star f_2)(g,g')=\kappa^2\int \dd z_1\dd z_3\dd z_2\dd z_4\ f_1(e^{-z_1}gg' e^{z_3})
f_2(e^{-z_2}gg'e^{z_4}) e^{-\frac{2i}{\theta}(\omega(z_1,z_2)-\omega(z_3,z_4))}.
\end{equation*}
Then, as in the previous section, we want to compute
\begin{multline*}
\Delta(f_1)\star \Delta(f_2)(g,g')=\kappa^4\int \dd z_1\dd z_3\dd z_2\dd z_4\dd z'_1\dd z'_3\dd z'_2\dd z'_4\ f_1(e^{-z_1}g e^{z_3-z'_1}g'e^{z'_3})\\
f_2(e^{-z_2}ge^{z_4-z'_2}g'e^{z'_4}) e^{-\frac{2i}{\theta}(\omega(z_1,z_2)-\omega(z_3,z_4)+\omega(z'_1,z'_2)-\omega(z'_3,z'_4))}.
\end{multline*}
For this, we change the variables $z''_3=z_3-z'_1$, $z''_4=z_4-z'_2$ and we perform the integration on $z'_1,z'_2$. After simplification, it gives the compatibility:
\begin{equation*}
\Delta(f_1)\star \Delta(f_2)(g,g')=\Delta(f_1\star f_2)(g,g').
\end{equation*}
Let us show for example the identity $\mu(\text{id}\otimes S)\Delta=1\otimes\eps$. Indeed,
\begin{multline*}
\mu(\text{id}\otimes S)\Delta(f)(g)=\kappa^2\int \dd z_1\dd z_3\dd z_2\dd z_4\ f\big(e^{-z_1}g e^{z_3}(e^{-z_2}ge^{z_4})^{-1}\big) e^{-\frac{2i}{\theta}(\omega(z_1,z_2)-\omega(z_3,z_4))}\\
=\kappa (-1)^n\int\dd z_1\dd z_2 f(e^{-z_1}gg^{-1}e^{z_2})e^{-\frac{2i}{\theta}\omega(z_1,z_2)}=f(e_G).
\end{multline*}
For the continuity of the coproduct, we need to choose a global odd coordinate system $\{\eta\}$ on $G$ since it is a trivial supermanifold (see Definition \ref{def-concrete-superman}). Then, the coproduct can be expressed as
\begin{equation*}
\Delta(f)(g,g')=f(g\fois g')=\sum_{I,J,L} c_{I,J,L} f_L((\gB g)(\gB g'))\eta^I(\eta')^J
\end{equation*}
if $g=(\gB g,\eta)$, $g'=(\gB g',\eta')$, and by denoting $c_{I,J,L}$ some coefficients related to the group law of $G$ and independent of $f$. Thus, we have the estimate
\begin{equation*}
|\Delta(f)|_{\alpha,K;\alpha',K'}\leq \sup_{g\in K,\,g'\in K',\,|\nu|\leq\alpha,\,|\nu'|\leq\alpha'} \sum_{I,J,L} |c_{I,J,L}| |D_g^\nu D_{g'}^{\nu'} f_L(g\fois g')| \leq C\sum_{\tau} |f|_{\tau,K''},
\end{equation*}
where $K''$ is a compact subset of $\gB G$ containing $\{g\fois g',\ g\in K,\ g'\in K'\}$, $\tau\leq\alpha+\alpha'$, and $C$ a constant depending in particular on $c_{I,J,L}$ and on the (smooth) modular function of $G$ and its derivatives. So, the coproduct is continuous, as well as the other operations.
\end{proof}

\begin{proposition}
The subgroup $\Gamma\subset G$ is not deformed in $H$. This means that $\Gamma=\gT^{m|n}$ is a subgroup of the quantum supergroup $(H,\mu,1,\Delta,\eps,S)$.
\end{proposition}
\begin{proof}
The coproduct is indeed not deformed. For the product, we can see that $\forall g\in\Gamma$, $\forall f_1,f_2\in H$,
\begin{equation*}
(f_1\star f_2)(g)=\kappa^2\int \dd z_1\dd z_3\dd z_2\dd z_4\ f_1(g e^{z_3-z_1})f_2(ge^{z_4-z_2}) e^{-\frac{2i}{\theta}(\omega(z_1,z_2)-\omega(z_3,z_4))}
\end{equation*}
since $\Gamma$ is Abelian. By performing the change of variables $z_1\mapsto z_1+z_3$, $z_2\mapsto z_2+z_4$ and integrating over $z_3,z_4$, we find $(f_1\star f_2)(g)=f_1(g)f_2(g)$. So, $\caC^\infty(\Gamma)$ is not deformed in $H$.
\end{proof}

The analog of Kac-Takesaki operator, defined in \eqref{eq-defmultunit}, has in this context the expression
\begin{equation}
W(f)(g,g')=\kappa^2\int \dd z_1\dd z_3\dd z_2\dd z_4\ f(e^{-z_1}gg' e^{z_3},e^{-z_2}g'e^{z_4}) e^{-\frac{2i}{\theta}(\omega(z_1,z_2)-\omega(z_3,z_4))},\label{eq-multunitsubtore}
\end{equation}
for $f\in H\widehat\otimes H$. As in Proposition \ref{prop-multunitsolv}, we can show that it is a continuous operator $W:H\widehat\otimes H\to H\widehat\otimes H$ homogeneous of degree 0 and that it satisfies the pentagonal equation
\begin{equation*}
W_{12}W_{13}W_{23}=W_{23}W_{12}.
\end{equation*}
Moreover if $G$ is unimodular, a computation analog as in Remark \ref{rmk-superunit} proves that $W$ is superunitary for the superhermitian scalar product canonically associated to $L^2(G\times G)$:
$\forall f_1,f_2\in (H\widehat\otimes H)\cap L^2(G\times G)$,
\begin{equation*}
\langle W(f_1),W(f_2)\rangle=\langle f_1,f_2\rangle.
\end{equation*}

\medskip

We finally give an explicit example of Fr\'echet quantum supergroup with a supertoral subgroup. For this, we will present the special linear supergroup in low dimension. We need to recall what supermatrices are.
\begin{definition}
\label{def-concrete-supermatrix}
A square supermatrix $A$ of size $(m|n)$ is a matrix with coefficients in $\superA$ (see section \ref{subsec-concrete}) and of the form
\begin{equation*}
A=\begin{pmatrix} A_{00} & A_{01} \\ A_{10} & A_{11} \end{pmatrix}
\end{equation*}
where $A_{00}$ is an $m\times m$ matrix with entries in $\superA_0$, $A_{01}$ is an $m\times n$ matrix with entries in $\superA_1$, $A_{10}$ is an $n\times m$ matrix with entries in $\superA_1$, and $A_{11}$ is an $n\times n$ matrix with entries in $\superA_0$.
\end{definition}
The set of square supermatrices of size $(m|n)$ is a superalgebra for the standard addition and multiplication. We denote by $GL(m|n)$ the supergroup of invertible square supermatrices of size $(m|n)$. Finally, we define the Berezinian (or superdeterminant) of a supermatrix $A$ by:
\begin{equation*}
\text{Ber}(A)=\det(A_{00}-A_{01} A_{11}^{-1}A_{10})\det(A_{11}^{-1}).
\end{equation*}
Now the special linear supergroup is defined as
\begin{equation*}
SL(m|n)\,:=\, \{\,A\in GL(m|n),\quad \text{Ber}(A)=1\}.
\end{equation*}
Restricted to dimension $m|n=1|1$, this supergroup contains the elements
\begin{equation*}
g=\begin{pmatrix} a & \beta \\ \gamma & d  \end{pmatrix}
\end{equation*}
with $a,d\in\superA_0$, $\beta,\gamma\in\superA_1$ such that $a=d+\frac{1}{d}\beta\gamma$. We see directly that $SL(1|1)$ contains two supertoral subgroups generated by $\beta\in\gT^{0|1}$ and $\gamma\in\gT^{0|1}$. We can choose for example to consider the deformation using the supertorus generated by $\beta$, and we want to see if this deformation is not trivial. For this, we compute the explicit expression of the star-product \eqref{eq-prodsubtore}: for any $f_1,f_2\in\caC^\infty(SL(1|1))$,
\begin{multline*}
(f_1\star f_2)(g)=f_1(g)f_2(g)+\frac{i\theta}{2}a\gamma (\partial_\beta f_1(g)\partial_d f_2(g)-\partial_d f_1(g)\partial_\beta f_2(g))\\
+\frac{i\theta}{2}d\gamma (\partial_\beta f_1(g)\partial_a f_2(g)-\partial_a f_1(g)\partial_\beta f_2(g))+\frac{i\theta}{2}(d^2-a^2)\partial_\beta f_1(g)\partial_\beta f_2(g).
\end{multline*}
We see that already by taking a supertoral subgroup of dimension $0|1$ we can produce a non-trivial Fr\'echet quantum supergroup, deformation of $SL(1|1)$. Note that this associative star-product stops at the finite level $\theta$ because only odd variables are involved in the deformation. This is a simple example that shows how such a construction can be useful in concrete cases. Of course, it applies on a large class of supergroups for which explicit expressions can be much more complicated.

\vskip 0.5 cm

\noindent {\bf Acknowledgements}: The author thanks Pierre Bieliavsky, Philippe Bonneau and Gijs Tuynman for interesting discussions on this work.

\end{document}